\documentclass[12pt]{article}

\usepackage{amsmath}
\usepackage{amstext}
\usepackage{amsthm}
\usepackage{amssymb}
\usepackage[dvips]{graphicx}
\usepackage{caption2}
\usepackage{amsfonts,mathrsfs}
\usepackage{makeidx}
\usepackage{fancyhdr}
\usepackage{epic,eepic,epsfig}
\usepackage{mdwlist}
\usepackage{amscd}
\usepackage{amsbsy}
\usepackage{epsfig,graphicx}
\usepackage{epsfig,wasysym,stmaryrd}
\usepackage[all]{xypic}
\usepackage{comment}
\usepackage{latexsym}

\usepackage{float}

\theoremstyle{plain}
\newtheorem{theorem}{Theorem}[section]

\newtheorem{lemma}[theorem]{Lemma}
\newtheorem{proposition}[theorem]{Proposition}

\theoremstyle{definition}

\theoremstyle{remark}

\def\Zbl{\textrm{Zbl\,}}
\def\MR{\textrm{MR\,}}

\begin{document}
\begin{center}
\textbf{Homology group of branched cyclic covering over a $2$-bridge knot of genus two}
\vspace{1cm}

\textbf{Ilya Mednykh}

\end{center}
\begin{abstract}
The structure of the first homology group of a cyclic covering of a knot is an important invariant well known in the knot theory. In the last century, H. Seifert developed a general approach to compute the homology group of the covering. Based  on his ideas R. Fox found explicit form for $H_{1}(M_{n},\mathbb{Z}),$ where $M_{n}$ is an $n$-fold cyclic covering over a knot $K$ admitting genus one Seifert surface. 

The aim of the present paper is to find the structure of $H_{1}(M_{n},\mathbb{Z})$ for $2$-bridge knots admitting genus two Seifert surface. The result is given explicitly in terms of Alexander polynomial of the knot.
\end{abstract}

\noindent\textbf{Primary:} 57M12; \textbf{Secondary:} 57K14, 39A06.

\noindent\textbf{Keywords:}
Alexander polynomial, knot, branched covering, homology group.

\section{Introduction}
Let $K$ be a knot in the three dimensional sphere $\mathbb{S}^{3}.$ Denote by $M_{n}$ a cyclic $n$-fold covering of $\mathbb{S}^{3}$ branched over $K.$ The general description of the first homology group $H_{1}(M_{n},\mathbb{Z})$ was done by Seifert \cite{Seifert}. To do this he introduced matrix $\Gamma$ which can be expressed via Seifert matrix $V$ of the knot $K$ as $\Gamma=(V-V^{\prime})^{-1}V$ where $V^{\prime}$ is a transpose of $V.$ Then $H_{1}(M_{n},\mathbb{Z})$ as an Abelian group is represented by the matrix $\Gamma^{n}-(\Gamma-I)^{n}.$

Recall that the Alexander polynomial of a knot $K$ is given by equality $A(z)=\det(V-z V^{\prime}).$ We note that even in the class of $2$-bridge knots Alexander polynomial does not determine the congruence class of matrix $V.$ For example, two $2$-bridge knots with rational slopes $17/2$ and $15/4$ share the same Alexander polynomial $A(z)=4-7z+4z^2,$ but have non-congruent Seifert matrices $\left(\begin{array}{cc} 4& 1 \\ 0 & 1\end{array}\right)$ and $\left(\begin{array}{cc} 2& 1 \\ 0 & 2\end{array}\right).$ In spite of this, these two knots have the same first homology group of a cyclic $n$-fold covering. See the Theorem~\ref{caseaiszero} below.

In general, the Alexander polynomial does not designate the structure of the first homology group $H_{1}(M_{n},\mathbb{Z}).$ Indeed, the stevedore's knot $6_{1}$ and the knot $9_{46}$ have the same Alexander polynomial, but in the first case, $H_1(M_{2},\mathbb{Z})\cong\mathbb{Z}_9,$ where as in the second, $H_1(M_{2},\mathbb{Z})\cong\mathbb{Z}_{3}\oplus\mathbb{Z}_{3}.$

However, if we restrict ourselves to the class of $2$-bridge knots, the situation became much more plausible. It was shown by A. Cattabriga and M. Mulazzani \cite{CattMull} that the polynomial associated with $n$-fold cyclic covering $M_{n}$ of a $2$-bridge knot coincides with the Alexander polynomial of this knot. That is, $H_{1}(M_{n},\mathbb{Z})$ is completely determined through Alexander polynomial of $2$-bridge knot. 

In this paper, we present explicit formulas for the group $H_{1}(M_{n},\mathbb{Z})$ for $2$-bridge knots of genus one and two by making use of coefficients of their Alexander polynomials.

\section{Cyclic $n$-fold covering of 2-bridge knots}

For the basic definitions and facts from the knot theory we refer the reader to the D. Rolfsen book \cite{Rolf} and A. Kawauchi survey \cite{Kawa}.

Recall the following basic notions and definitions. A $2$-bridge knot (or rational) is a knot admitting a diagram with two bridges. A full classification of $2$-bridge knots was done by H. Schubert \cite{Schu}. These knots are uniquely defined by their rational slope $p/q,$ where $p$ and $q$ are positive integers and $1\le q \le p-1.$ Two knots with rational slopes $p/q,$ and $p^\prime/q^\prime,$ are equivalent if and only if $p=p^\prime$ and $q\equiv-q^\prime$ or $q\,q^\prime\equiv 1\mod p.$ By M. Boileau and B. Zimmermann \cite{BoilZimm}, a knot is $2$-bridge if and only if the fundamental group of it complement in the three dimensional sphere is generated by two meridians.
  
The general approach to calculate Conway and Alexander polynomials of a $2$-bridge knot is given in \cite{KoselPeck}. 
 
A Seifert surface of a knot $K$ is an orientable surface with one boundary component which coincides with the knot $K.$

A genus of the knot $K$ is a minimum genus of any Seifert surface for $K.$
  
The degree Alexander polynomial of genus $g$ knot never exceed $2g$ \cite{Seifert}.

Let $K$ be a $2$-bridge knot and $M_{n}$ is an $n$-fold cyclic covering of the three dimensional sphere branched over $K.$ We note that the fundamental group $\pi_1(\mathbb{S}^3\setminus K)$  is generated by two meridians $\texttt{x}_0,\,\texttt{x}_1$ of knot $K$  and has  exactly one relation  $r_1(\texttt{x}_0,\,\texttt{x}_1)=0$  whose structure is completely determined by the slope of the knot \cite{Schu}. Denote by $A(z)=\sum_{\ell=0}^{s}a_{\ell}z^{\ell}$ the Alexander polynomial of  knot $K$ normalizied by conditions $a_{0}\neq0,\,a_{s}\neq0$ and $A(1)=1.$ According to (\cite{Fox3}, p. 416) the first homology group of $M_{n}$ is an Abelian group represented by matrix $A(T_{n}),$ where $T_{n}=circ(0,1,\ldots,0)$ is an $n\times n$ circulant matrix. Thereby, $\mathcal{H}=H_{1}(M_{n},\mathbb{Z})$ has the following representation $$\mathcal{H}=\big\langle x_{1},\,x_{2},\ldots,x_{n}|A(T_{n})(x_{1},x_{2},\ldots,x_{n})^{t}=0\big\rangle.$$ This leads to equivalent infinite representations 
$$\mathcal{H}=\big\langle x_{j},\,j\in\mathbb{Z}|\sum_{\ell=0}^{s}a_{\ell}x_{i+\ell}=0,\,x_{i+n}=x_{i},\,i\in\mathbb{Z}\big\rangle$$ 
We set $$U=\left(\begin{array}{cccc}0&-1& \ldots &0 \\ 0&0& &0\\ \vdots & & \ddots& \vdots \\ 0&0& &-1 \\ a_{0}&a_{1}& \ldots & a_{s-1}\end{array}\right) \text {  and   }  V=\left(\begin{array}{ccccc}1&0& \ldots &0&0 \\ 0&1& &0&0\\ \vdots & & \ddots& & \vdots \\ 0&0& &1&0 \\ 0&0& \ldots & 0&a_{s}\end{array}\right).$$

Then  we can rewrite the representation of group $\mathcal{H}$ as 

$$\mathcal{H}=\big\langle x_{j},\,j\in\mathbb{Z}|\,
U\,X_{i}+V\,X_{i+1}=0,X_{i+n}=X_{i},\,i\in\mathbb{Z}\big\rangle,$$ where $X_{i}=(x_{i},x_{i+1},\ldots,x_{i+s-1})^{t}$ is a column vector of length $s.$ 
We note that $U+V=\left(\begin{array}{ccccc}1&-1& \ldots &0&0 \\ 0&1& &0&0\\ \vdots & & \ddots& & \vdots \\ 0&0& &1& -1 \\ a_{0}& a_{1} & \ldots & a_{s-2}& a_{s-1}+a_{s}\end{array}\right).$ By induction we have $$\det(U+V)=a_{0}+a_{1}+\ldots+a_{s-1}+a_{s}=A(1)=1.$$ Therefore both $U+V$ and $(U+V)^{-1}$ are integer unimodular matrices. Next,  we multiply each of the relations $U\,X_{i}+V\,X_{i+1}=0,\,i\in\mathbb{Z}$ from left-hand side by $(U+V)^{-1}$ to obtain $$\mathcal{H}=\big\langle x_{j},\,j\in\mathbb{Z}|
\Gamma\,X_{i}-(\Gamma-I)\,X_{i+1}=0,X_{i+n}=X_{i},\,i\in\mathbb{Z}\big\rangle,$$ where $\Gamma=(U+V)^{-1}U.$ 

We finish the section with the following lemma attributed to H. Seifert (\cite{Seifert}, p. 577, Satz 1).

\begin{lemma}\label{Seifertlemma} Let $K$ be a two-bridge knot with the Alexander polynomial $A(z).$ Denote by $\mathcal{H}=H_{1}(M_{n},\mathbb{Z})$ the first homology group of the cyclic $n$-fold covering $M_{n}$ of the three-dimensional sphere $\mathbb{S}^3$ branched over $K.$ Then the Abelian group $\mathcal{H}$ has the following matrix of relations $\Gamma^{n}-(\Gamma-I)^{n}.$  
\end{lemma}

One can consider the matrices $U$ and $V$ as combinatorial analogues of the Seifert matrix introduced in \cite{Seifert}. The matrix $\Gamma$ was also introduced by Seifert as $\Gamma=(U+V)^{-1}U$ by setting $V=-U^{t}.$ We also note that the Alexander polynomial can be expressed as $A(z)=\det(U+z\,V).$

\section{Preliminary observations}

Consider a $2$-bridge knot $K$ whose Alexander polynomial has the form $A(z)=a+(b-a)z+(1-2b)z^2+(b-a)z^3+a z^4.$ This is a general form for a knot of Seifert genus $2.$ (See, for example, \cite{Fox}). Denote by $M_{n}$ a cyclic $n$-fold covering of the three sphere $\mathbb{S}^3$  branched over knot $K.$ The aim of this paper is to describe explicitly the structure of the first homology group  $H_1(M_n,\mathbb{Z})$ for any natural $n.$

We refer to the content of the previous section to get the following combinatorial version of Seifert matrix
$$\Gamma=\left(\begin{array}{cccc}a& -1 + b& - b& -a\\a& b& -b& -a\\a& b& 1 - b& -a\\a& b& 1 - b& 1 - a\end{array}\right).$$
By making use of Lemma~\ref{Seifertlemma}, we conclude that the Abelian group $H_{1}(M_{n},\mathbb{Z})$ is represented by the matrix 
\begin{equation}\label{defineB}B(n)=\Gamma^{n}-(\Gamma-I)^{n}.\end{equation} 
See also \cite{Fox} for more detailed explanation. One can consider matrix $B(n)$ as $\mathbb{Z}$-linear operator from $\mathbb{Z}^4$ to $\mathbb{Z}^4$ acting by the rule $B(n):x\rightarrow B(n)x,\,x\in\mathbb{Z}^4.$ Define $\textrm{coker}\,B(n)$ as the factor space $\mathbb{Z}^4/\textrm{im}\,B(n),$  where  $\textrm{im}\,B(n)$ is the image of $B(n)$ in $\mathbb{Z}^4.$ Then one has
$$H_{1}(M_{n},\mathbb{Z})=\textrm{coker}\,B(n).$$

The structure of the Abelian group $\textrm{coker}\,B(n)$ is completely determined by the Smith Normal Form of matrix $B(n).$ More precisely,

$$\textrm{coker}\,B(n)=\mathbb{Z}_{d_1 }\oplus\mathbb{Z}_{{d_2}/{d_1}}\oplus\mathbb{Z}_{{d_3}/{d_2}}\oplus\mathbb{Z}_{{d_4}/{d_3}},$$ 
where $d_k$ is the greatest common divisor of  all $k\times k$  minors of matrix  $B(n).$
By definition, $\mathbb{Z}_{d}$ is the factor group $\mathbb{Z}/d\,\mathbb{Z},$ where $d$ is an  integer number. It can be positive, negative or zero. In particular, $\mathbb{Z}_{0}=\mathbb{Z}$ and 
$\mathbb{Z}_{-5}=\mathbb{Z}_{5}.$

\section{Recurrent sequences associated with Alexander polynomial.}

Consider the matrix $\Gamma$ from previous section. Characteristic polynomials for $\Gamma$ and $\Gamma-I$ are $P_{0}(\lambda)=a-(3a+b)\lambda+(1+3a+b)\lambda^2-2\lambda^3+\lambda^4$ and $P_{1}(\lambda)=a+(3a+b)\lambda+(1+3a+b)\lambda^2+2\lambda^3+\lambda^4$ respectively. Their product is 
$P_{0}(\lambda)P_{1}(\lambda)=a^{2}+(2a-3a^2-4a b-b^2)\lambda^{2}+(1-4a+9a^2-2b+6a b+b^2)\lambda^{4}+(-2+6a+2b)\lambda^{6}+\lambda^{8}.$ 

Hence, each entry of matrix $B(n)$ and the matrix $B(n)$ itself satisfy the following recursion 

\begin{align}\label{mainrecurs}
\nonumber &a^2 u(n) + (2 a - 3 a^2 - 4 a b - b^2) u(n+2) \\
&+ (1 - 4 a + 9 a^2 - 2 b + 6 a b + b^2) u(n+4) \\
\nonumber &+ (-2 + 6 a + 2 b) u(n+6) + u(n+8)=0.
\end{align}

We make the following important observation. Consider eigenvalues $\lambda_1,\lambda_2,\lambda_3,\lambda_4$ of the matrix $\Gamma.$ They are $$\lambda_{1,2}=\frac{1}{2}(1\pm\sqrt{1-6a-2b-2\sqrt{-4a+(3a+b)^2}})$$ and $$\lambda_{3,4}=\frac{1}{2}(1\pm\sqrt{1-6a-2b+2\sqrt{-4a+(3a+b)^2}}).$$ The list of eigenvalues of $\Gamma-I$ coincides with $\{-\lambda_1, -\lambda_2, -\lambda_3, -\lambda_4\},$ where $\{\lambda_1,\lambda_2,\lambda_3,\lambda_4\}$ are eigenvalues of $\Gamma.$

Now we introduce four integer sequences.  
$$s_{even}(n)=\frac{\lambda_2^n-\lambda_1^n}{\lambda_2-\lambda_1}+\frac{\lambda_4^n-\lambda_3^n}{\lambda_4-\lambda_3},$$ 

$$t_{even}(n)=\left(\frac{\lambda_2^n-\lambda_1^n}{\lambda_2-\lambda_1}-\frac{\lambda_4^n-\lambda_3^n}{\lambda_4-\lambda_3}\right)/\left(\Big(\frac{\lambda_4-\lambda_3}{2}\Big)^2-\Big(\frac{\lambda_2-\lambda_1}{2}\Big)^2\right),$$ 

$$s_{odd}(n)=\lambda_1^n+\lambda_2^n+\lambda_3^n+\lambda_4^n,$$ 

$$t_{odd}(n)=(\lambda_1^n+\lambda_2^n-\lambda_3^n-\lambda_4^n)/\left(\Big(\frac{\lambda_4-\lambda_3}{2}\Big)^2-\Big(\frac{\lambda_2-\lambda_1}{2}\Big)^2\right).$$  

We note that $\left((\frac{\lambda_4-\lambda_3}{2})^2-(\frac{\lambda_2-\lambda_1}{2})^2\right)=\sqrt{-4a+(3a+b)^2}.$

The fact that all of them are expressed as linear combinations  of $n$-th powers of eigenvalues of $\Gamma$ and $\Gamma-I$ implies that they satisfy the same recursive relation~(\ref{mainrecurs}).  

Then we also introduce two sequences. $$s(n)=\begin{cases} s_{even}(n),\textrm{ if }n\textrm{ is even}\\ s_{odd}(n),\textrm{ if }n\textrm{ is odd}\end{cases}\textrm{ and }\ t(n)=\begin{cases}t_{even}(n),\textrm{ if }n\textrm{ is even}\\ t_{odd}(n),\textrm{ if }n\textrm{ is odd}\end{cases}.$$  

The structure of matrix $B(n)$ is given by the following lemmas. 

\begin{lemma}\label{oddformB} 
Suppose $n$ is odd. Then the matrix $B(n)=s(n)L+t(n)R$ is a sum of two centrosymmetric matrices, where $L$ and $R$ are
$$L=\left(\begin{array}{cccc}
\frac{1}{2}& 0& 0& 0 \\ 
0& \frac{1}{2}& 0& 0 \\ 
0 & 0 &\frac{1}{2}& 0 \\
0 & 0 & 0&\frac{1}{2}
\end{array}\right)$$
and 
$$R=\left(\begin{array}{cccc}
\frac{a-b}{2}& -1+a+2b& a-b& -a\\ 
a & \frac{-a+b}{2}& a & 0\\ 
0 & a& \frac{-a+b}{2}& a\\ 
-a & a-b &-1+a+2b & \frac{a-b}{2}
\end{array}\right).$$
\end{lemma}

\begin{proof} We note that both sides of equality $B(n)=s(n)L+t(n)R$ satisfy the linear recursive
equation~(\ref{mainrecurs}). So, it is sufficient to prove Lemma for $n=1,3,5,7.$
By making use of symbolic calculation with Wolfram Mathematica 12.0 we deduce that 
\begin{align}\label{Basedatasodd}
\nonumber s(1)&=2,\,s(3)=2-9a-3b,\\ 
s(5)&=2-25a+45a^2-5b+30a b+5b^2,\\
\nonumber s(7)&=2-49a+189a^2-189a^3-7b+105a b\\
\nonumber&-189a^2b+14b^2-63a b^2-7b^3
\end{align} 
and  
\begin{align}\label{Basedatatodd}
\nonumber &t(1)=0,\,t(3)=-3,\\
&t(5)=5(-1+3a+b),\\
\nonumber &t(7)=-7(1-7a+9a^2-2b+6a b+b^2).
\end{align}

Then the result follows by the direct substitution.      
\end{proof}

\begin{lemma}\label{evenformB} 
Suppose $n$ is even. Then the matrix $B(n)=s(n)L+t(n)R$ is a sum of two anticentrosymmetric matrices, where $L$ and $R$ are   
$$L=\left(\begin{array}{cccc}
\frac{-1+2a}{2}& -1+b& -b& -a\\ 
a& \frac{-1+2b}{2}&-b&-a\\ 
a & b &\frac{1-2b}{2}&-a\\ 
a & b & 1-b &\frac{1-2a}{2}
\end{array}\right)$$
and $R$ is block matrix of the form 
$\left(\begin{array}{cc}
L & M\\ 
\hat{M} & \hat{L}
\end{array}\right),$ where 
$$L=\left(\begin{array}{cc}
\frac{-5a+6a^2+b+2a b}{2}& 1-2a-3b+3a b+b^2\\ 
a(-1+3a+b)& \frac{(3a+b)(-1+2b)}{2}
\end{array}\right),$$  

$$M=\left(\begin{array}{cc}
-a+b-3a b-b^2& -a(-1+3a+b)\\ 
a-3a b-b^2 & -a(3a+b)
\end{array}\right)$$ and $\widehat{\left(\begin{array}{cc}
x& y\\ 
z& w
\end{array}\right)}=\left(\begin{array}{cc}
-w& -z\\ 
-y& -x
\end{array}\right).$ 
\end{lemma}

\begin{proof} Since matrix $B(n)=s(n)L+t(n)R$ is a solution of linear recursive equation~(\ref{mainrecurs}), it is sufficient to prove Lemma for $n=0,2,4,6.$
We have 
\begin{align}\label{Basedataseven}
\nonumber &s(0)=0,\,s(2)=2,\\
&s(4)=2(1-3a-b),\\ 
\nonumber &s(6)=2 - 18 a + 27 a^2 - 4 b + 18 a b + 3 b^2
\end{align} 
and  
\begin{equation}\label{Basedatateven}
t(0)=0,\,t(2)=0,\,t(4)=-2,\,t(6)=-4+9a+3b.
\end{equation}

Then the result follows by the direct calculations.   
\end{proof}
\bigskip

In what follows, we will call the sequences $s(n)$ and $t(n)$ to be \textit{associated} with Alexander polynomial $A(z).$ They are unique solutions of  of linear recursive equation~(\ref{mainrecurs})  with initial data (\ref{Basedatasodd}) -- (\ref{Basedatateven}).

\section{Prelminary results}

The next theorem is reformulation of a result previously proved by various authors (\cite{Fox}, \cite{Mul}, \cite{MulVes}).

\begin{theorem}\label{caseaiszero} Let $K$ be a $2$-bridge knot of genus $1$ with Alexander polynomial $A(z)=b+(1-2b)z+b\,z^2.$ We set $\alpha(n)=\left(\frac{1+x}{2}\right)^n+\left(\frac{1-x}{2}\right)^n$ and $\beta(n)=\frac{1}{x}\left(\left(\frac{1+x}{2}\right)^n-\left(\frac{1-x}{2}\right)^n\right),$ where $x=\sqrt{1-4b}.$  

Then the first homology group of cyclic $n$-fold covering $M_{n}$ of the three sphere $\mathbb{S}^3$ branched over knot $K$ has the following structure

$$H_1(M_n,\mathbb{Z})=\begin{cases}\mathbb{Z}_{\alpha(n)}\oplus\mathbb{Z}_{\alpha(n)}\textit{ if }n\textit{ is odd,}\\ \mathbb{Z}_{\beta(n)}\oplus\mathbb{Z}_{(4b-1)\beta(n)}\textit{ if }n\textit{ is even.}\end{cases}$$
\end{theorem}

\begin{proof}
By K. Funcke (\cite{Funke}, Sats 5.1), any $2$-bridge knot of genus $1$ has a slope $p/q=2\lambda_{1}+\frac{1}{-2\lambda_{2}}$ for some integers $\lambda_{1},\lambda_{2}.$ In this case, Seifert matrix of the knot has a form $V=\left(\begin{array}{cc} \lambda_{1}& 1 \\ 0 & \lambda_{2}\end{array}\right).$ One can recognize Alexander polynomial $A(z)$ as determinant of the matrix $V-zV^{t},$ where $V^{t}$ is transpose of matrix $V.$ Therefore, $A(z)=\lambda_{1}\lambda_{2}+(1-2\lambda_{1}\lambda_{2})z+\lambda_{1}\lambda_{2}z^2.$ That is $b=\lambda_{1}\lambda_{2}.$ 

Following Seifert we set $\Gamma=(V-V^{t})^{-1}V.$ Then by Seifert theorem (\cite{Seifert}, Satz 1) $H_1(M_n,\mathbb{Z})$ is an Abelian group given by the matrix $H(n)=\Gamma^n-(\Gamma-\mathtt{1})^n.$ By direct calculations we have $H(n)=\left(\begin{array}{cc} \alpha(n) & 0\\ 0& \alpha(n)\end{array}\right)$ if $n$ is odd and $H(n)=\beta(n)\left(\begin{array}{cc} 1& 2 \lambda_{2}\\ -2 \lambda_{1}& -1\end{array}\right)$ if $n$ is even. Since the Smith Normal Form of matrix $\left(\begin{array}{cc} 1& 2 \lambda_{2}\\ -2 \lambda_{1}& -1\end{array}\right)$ is $\left(\begin{array}{cc} 1& 0 \\ 0& 4\lambda_{1}\lambda_{2}-1\end{array}\right)$ the result follows.
\end{proof}
\bigskip

Theorem~\ref{caseaiszero} shows that the structure of the first homology group $H_1(M_n,\mathbb{Z})$ completely defined by the Alexander polynomial of the $2$-bridge knot $K.$ This fails to be true for $3$-bridge knots. Indeed, consider two knots $6_{1}$ and $9_{46}$ shown on figures  $50$ and $51$ in the book \cite{Fox}. The first one is a $2$-bridge knot while the other is a  $3$-bridge knot. They share the same Alexander polynomial $A(z)=-2+5z-2z^2$ but have the different Seifert matrices $\left(\begin{array}{cc} 1& 1 \\ 0& -2\end{array}\right)$ and $\left(\begin{array}{cccc}1& 0& 0& 0\\ 0&-1& 0& 0\\  1& 0& 1& 1\\ -1& -1& 0& 0\end{array}\right)$ respectively. Consequently, the first homology group of their cyclic coverings are 
$$6_{1}:\begin{cases}\mathbb{Z}_{\alpha(n)}\oplus\mathbb{Z}_{\alpha(n)}\textit{ if n is odd,}\\ \mathbb{Z}_{\beta(n)}\oplus\mathbb{Z}_{9\beta(n)}\textit{ if n is even,}\end{cases}\textrm{ and    }9_{46}:\begin{cases}\mathbb{Z}_{\alpha(n)}\oplus\mathbb{Z}_{\alpha(n)}\textit{ if n is odd,}\\ \mathbb{Z}_{3\beta(n)}\oplus\mathbb{Z}_{3\beta(n)}\textit{ if n is even.}\end{cases}$$  

Here $\alpha(n)=2^{n}+(-1)^n$ and $\beta(n)=\frac{1}{3}(2^{n}-(-1)^n).$

\section{Main results}

Our first result is given by the following theorem.

\begin{theorem}\label{HomoddB} Let $K$ be a $2$-bridge knot with Alexander polynomial $A(z)=a+(b-a)z+(1-2b)z^2+(b-a)z^3+a z^4.$ Let $s(n)$ and $t(n)$ be the associated with $A(z)$ recurrent sequences. Suppose that $n$ is odd. Consider the two following sequences $\hat\alpha(n)=\gcd(t(n),\frac{s(n)+(5a-b)t(n)}{2})$ and $\hat\beta(n)=\frac{s(n)^2+(4a-(3a+b)^2)t(n)^2}{4\hat\alpha(n)}.$

Then the first homology group of cyclic $n$-fold covering $M_{n}$ of the three sphere $\mathbb{S}^3$ branched over knot $K$ has the following structure
$$H_1(M_n,\mathbb{Z})=\mathbb{A}\oplus\mathbb{A},\textit{ where }\mathbb{A}=\mathbb{Z}_{\hat\alpha(n)}\oplus\mathbb{Z}_{\hat\beta(n)}.$$
\end{theorem}

\begin{proof}
Taking into account lemma \ref{oddformB} we conclude that matrix $B(n)$ which represents the first homology group $H_1(M_n,\mathbb{Z})$ has the form 
$$B(n)=\left(\begin{array}{cccc}
\alpha& \beta & \gamma & \delta\\ 
-\delta& \alpha-\gamma& -\delta& 0\\ 
0 & -\delta & \alpha-\gamma &-\delta\\ 
\delta & \gamma & \beta &\alpha
\end{array}\right),$$ where $\alpha=\frac{s(n)+(a-b)t(n)}{2}, \beta=(-1+a+2b)t(n), \gamma=(a-b)t(n), \delta=-a\,t(n).$ To describe the structure of the first homology group of we have to find the Smith Normal Form of its relation matrix $B(n).$ The following sequence of elementary row-column transformations significantly simplify our task. 
$$\left(\begin{array}{cccc}
\alpha& \beta & \gamma & \delta\\ 
-\delta& \alpha-\gamma& -\delta& 0\\ 
0 & -\delta & \alpha-\gamma &-\delta\\ 
\delta & \gamma & \beta &\alpha
\end{array}\right)\to
\left(\begin{array}{cccc}
\alpha& \beta-\delta & \gamma & \delta\\ 
-\delta& \alpha-\gamma& -\delta& 0\\ 
0 & 0 & \alpha-\gamma &-\delta\\ 
\delta & \gamma-\alpha & \beta &\alpha
\end{array}\right)\to$$

$$\left(\begin{array}{cccc}
\alpha& \beta-\delta & \gamma-\alpha & \delta\\ 
-\delta& \alpha-\gamma& 0 & 0\\ 
0 & 0 & \alpha-\gamma & -\delta\\ 
\delta & \gamma-\alpha & \beta-\delta &\alpha
\end{array}\right)\to
\left(\begin{array}{cccc}
\alpha& \beta-\delta & 0 & 0\\ 
-\delta& \alpha-\gamma& 0 & 0\\ 
0 & 0 & \alpha-\gamma & -\delta\\ 
0 & 0 & \beta-\delta & \alpha
\end{array}\right).$$
 
So $B(n)$ splits into two $2\times2$ blocs 
$\left(\begin{array}{cc}
\alpha& \beta-\delta \\ 
-\delta& \alpha-\gamma 
\end{array}\right)$ and $\left(\begin{array}{cc}
\alpha-\gamma & -\delta\\ 
\beta-\delta & \alpha
\end{array}\right).$ They have the same Smith Normal Form. The block 
$\left(\begin{array}{cc}\alpha& \beta-\delta \\ 
-\delta& \alpha-\gamma \end{array}\right)$ is equivalent to $\left(\begin{array}{cc}d_1& 0 \\  0 & d_2 \end{array}\right),$ where $d_1=\gcd(\alpha,\,\beta-\delta,\,-\delta,\,\alpha-\gamma)=\gcd(\alpha,\,\beta,\,\gamma,\,\delta)$ and $d_2=\det\left(\begin{array}{cc}\alpha& \beta-\delta \\ 
-\delta& \alpha-\gamma \end{array}\right)/d_1=\frac{\alpha^2-\alpha\,\gamma+\beta\,\delta-\delta^2}{d_1}.$ After substituting $\alpha,\,\beta,\,\gamma$ and $\delta$ with their expressions through $a,\,b,\,s,\,t$ we get $d_1=\gcd(t(n),\frac{s(n)+(5a-b)t(n)}{2})$ and $d_2=\frac{s(n)^2+(4a-(3a+b)^2)t(n)^2}{4d_1}.$ Finally, the Smith Normal Form of $B(n)$ is diagonal matrix $\textrm{diag}(d_1,\,d_1,\,d_2,\,d_2).$ 
\end{proof}

The second result described the structure of the first homology group for the case of even $n.$ For the sake of simplicity we will write $s,\,t,\,\mu,\,\zeta$ instead of $s(n),\,t(n),\,\mu(n),\,\zeta(n).$ 

\begin{theorem}\label{HomevenB}
Let $K$ be a knot with Alexander polynomial $A(z)=a+(b-a)z+(1-2b)z^2+(b-a)z^3+a z^4.$ Let $s=s(n)$ and $t=t(n)$ be the associated with $A(z)$ recurrent sequences. Suppose that $n$ is even. We set $k=1+4a-4b,\,\zeta=\frac{s+(5a-b)t}{2}$ and $\mu=\frac{s^2+(4a-(3a+b)^2)t^2}{4}.$ Then the first homology group of cyclic $n$-fold covering $M_{n}$ of the three sphere $\mathbb{S}^3$ branched over knot $K$ has the following structure 
$$H_1(M_{n},\mathbb{Z})=\mathbb{Z}_{d_1}\oplus\mathbb{Z}_{d_2/d_1}\oplus\mathbb{Z}_{d_3/d_2}\oplus\mathbb{Z}_{d_4/d_3},$$
where
$d_1=\gcd(t,\,\zeta), d_2=\gcd(\mu,\,\zeta\,t, k\,t^2), d_3=\mu\,\gcd(k\,t, \zeta)$ and $d_4 = k\,\mu^2.$
\end{theorem}

\begin{proof}
We note that $$\det(B(n))=(1+4a-4b)(\frac{s^2+(4a-(3a+b)^2)t^2}{4})^2=k\,\mu^2,$$ where
$k=1+4a-4b$  and $\mu=\frac{s^2+(4a-(3a+b)^2)t^2}{4}.$

The number $k=A(-1)$ is well known in the knot theory. It absolute value is called a determinant of the knot  $K$  and was introduced firstly by K. Reidemeister in his book \cite{Reid}.

Since $k=1+4a-4b$ is odd we conclude that $\mu$ is an integer number. Moreover, since both sequences $s=s(n)$ and $t=t(n)$ satisfy the recursive relation (\ref{mainrecurs}) the same holds for $\zeta(n)=\frac{s(n)+(5a-b)t(n)}{2}.$ Taking into account the initial data for $s(n)$ and $t(n)$ we have $\zeta(0)=0,\,\zeta(2)=1,\,\zeta(4)=1-8a,\,\zeta(6)=1-19a+36a^2+12a\,b.$ Hence, $\zeta(n)$ is integer number for all even $n.$
 
As we have expression of the matrix $B=B(n)$ from lemma \ref{evenformB}, one can easily find two integer matrices $C$ and $D$ such that $B=\zeta\,C+t\,D.$ So, each entry $B_{i,j}$ of matrix $B$ is an integer linear combination $\zeta$ and $t.$ The converse  is also true. Namely, $\zeta=-B_{1,1}-B_{2,4}$ and 
$t=(-2-8a+2b+4a\,b)B_{1,1}+B_{1, 2}+2(2-b)B_{1,3}+(5-8a+6b+4a\,b-4b^2)B_{1,4}-2(2-b)B_{2,3}-(3+6b-4b^2)B_{2, 4}.$ This leads to the following result. 

\begin{proposition}\label{evend1}  
The greatest common divisor $d_1$ of all entries of matrix $B$ is equal to the greatest common divisor of $t$ and $\zeta.$
\end{proposition}

In order to find the second element of the Smith Normal Form of $B$ we need to find the greatest common divisor of all $2\times2$ minors of $B.$   
 
\begin{proposition}\label{evend2}  
The greatest common divisor $d_2$ of all $2\times2$ minors of matrix $B$ is equal to   $\gcd(\mu,\,\zeta\,t,\,k\,t^2).$
\end{proposition} 

\begin{proof}
Following notation from Wolfram Mathematica \cite{mathem} denote by $P=\textrm{Minors[B,\,2]}$ the six by six matrix consisting of $2\times2$ minors of $B.$ By $P_{i,j}$ we denote the $(i,\,j)$-entry of matrix $P.$

By making use of lemma \ref{evenformB} one can find three integer matrices $E,\,F$ and $G$ such that $P=\mu\,E+t\,\zeta\,F+k\,t^2\,G.$ This can be done by the following procedure. First, we express each $P_{i,j}$ in the form $P_{i,j}=\mu\,E_{i,j}+R_{i,j},$ where $E_{i,j}$ and $R_{i,j}$ are the polynomial quotient and polynomial remainder of $P_{i,j}$ and $\mu$ with respect to $s.$ Then we set $E=\{E_{i,j}\}_{i,j=1,\ldots,6}$ and note that all the entries of $E$ are integer polynomials in $a$ and $b.$ Next, we present each $R_{i,j}$ as $R_{i,j}=t\,\zeta\,F_{i,j}+k\,t^2\,G_{i,j},$ where $F_{i,j}$ and $k\,t^2\,G_{i,j}$ are the polynomial quotient and polynomial remainder of $R_{i,j}$ and $t\,\zeta$ with respect to $s.$ Again, all the entries of matrices $F=\{F_{i,j}\}_{i,j=1,\ldots,6}$ and $G=\{G_{i,j}\}_{i,j=1,\ldots,6}$ are integer polynomials in $a$ and $b.$  

Set $x=\mu,\,y=t\,\zeta$ and $z=k\,t^2.$ From now on we can consider the elements $P_{i,j}$ as integer linear combinations of $x,\,y$ and $z$ whose coefficients are integer polynomials in $a$ and $b.$ 

Consider $J=Span_{\mathbb{Z}}(P_{i,j},\,i,j=1,\ldots,6)$ as a $\mathbb{Z}$-module. As it shown before, each $P_{i,j}$ is an integer linear combination of $x,\,y$ and $z.$ Hence the same is true for all elements of $J.$ Now we will prove that $J$ is generated by $x,\,y,\,z$ as a $\mathbb{Z}$-module. This is equivalent to the statement of proposition under proof. To do this we need only to show that all three of $x,\,y,\,z$ belong to $J.$ The inclusion of $y$ in $J$ follows from the identity  

\begin{eqnarray*}
y&=&2(-1-6a+2a b)(2P_{1, 3}+P_{1, 6}- P_{3, 3}) + 2(2-b)(P_{6, 5}+P_{6, 6})\\
&+&P_{3, 5} + (4+a-7b+2b^2)(2P_{1, 5}-P_{2, 5})\\
&+&(15+14a-28b-4a b+8b^2) (P_{4, 3}-P_{5, 1}+P_{6, 3}).
\end{eqnarray*}

To prove that $x$ and $z$ lie in $J$ we make some small observation 

\begin{eqnarray*}
&P_{1, 3}&-P_{1, 6} = 2a\,x,\\ 
&P_{4, 3}&-P_{5, 1}+P_{6, 3} = a\,y,\\
&2P_{1, 5}&-P_{2, 5} = -a\,z+2a\,y.
\end{eqnarray*}

Hence, $2a\,x,\,a\,y$ and $a\,z$ lie in $J.$ 

Next, we take the following $4$-tuple of elements of $J:(P_{1, 1} + P_{1, 2}, P_{1, 4}, P_{2, 4}, P_{3, 4}).$ Reducing each element of this tuple modulo polynomials $2a\,x,\,a\,z$ and $y$ we get much more simple expression whose parts belong to $J:$

$$(x, -b^2 z+2b x, b z - 2b^2 z + (-2 + 6b) x, -z + 4b z - 3b^2 z + (-4 + 8b)x).$$ 

So, consequently we conclude that $x\in J,\,b^2 z \in J,\, b\,z\in J$ and lastly, $z\in J.$ 
\end{proof}

\begin{proposition}\label{evend3}  
The greatest common divisor $d_3$ of all $3\times3$ minors of matrix $B$ is equal to $\mu\,\gcd(k\,t, \zeta).$
\end{proposition}

\begin{proof}
As before, following notation from Wolfram Mathematica \cite{mathem} denote by $P^{\prime}=\textrm{Minors[B,\,3]}$ the four by four matrix consisting of $3\times3$ minors of $B.$ By $P_{i,j}^{\prime}$ we denote the $(i,\,j)$ -entry of matrix $P^{\prime}.$

From lemma \ref{evenformB} one can derive the following identity 
$$P^{\prime}=\mu(k\,t\,C^\prime+\zeta\,D^\prime),$$
where 
$$C^\prime=\left(\begin{array}{cccc}
-a + b& a& 0& a\\ -1-a+2b& 0& a& a-b\\ 
-a+b& -a& 0& 1+a-2 b\\ -a& 0& -a& a-b
\end{array}\right)$$ and 
$$D^\prime=\left(\begin{array}{cccc}
1 + 2 a - 4 b& -2 a& -2 a& -2 a\\ 
2 + 4 a - 6 b& 1 - 2 b& -4 a + 2 b& -4 a + 2 b\\
4 a - 2 b& 4 a - 2 b& -1 + 2 b& -2 - 4 a + 6 b\\
2 a& 2 a& 2 a& -1 - 2 a + 4 b
\end{array}\right).$$ So, all $3\times3$ minors of $B$ are linear combinations of the form $\alpha\,k\,t\,\mu+\beta\,\zeta\,\mu.$ Moreover, one can see that $\zeta\,\mu=P_{2,2}^\prime+P_{2,3}^\prime-P_{1, 2}^\prime-P_{1, 3}^\prime$ and $k\,t\,\mu=(-1+4a+2b)P_{1,2}^\prime+(-3+4a+2b)P_{1,3}^\prime-2(-1+2a+b) (P_{2,2}^\prime+P_{2,3}^\prime)-2P_{2, 4}^\prime+P_{3, 4}^\prime.$ Hence, the greatest common divisor of all $3\times3$ minors of $B$ coincides with the greatest common divisor $\gcd(k\,t\,\mu,\zeta\,\mu)=\mu\gcd(k\,t,\zeta).$
\end{proof}

There is only one $4\time4$ minor of $B$ which is $\det(B)=k\,\mu^2.$ Hence, the  greatest common divisor $d_4$ of all $4\times 4$ minors is equal to $k\,\mu^2.$ Now, we use all the above propositions to finish the proof of the theorem.
\end{proof}

The connection between Theorem 1 and Theorem 2 is given by the following refined version of the second Plans theorem (\cite{Plans}, \cite{Steve}).

\begin{theorem} Let $K$ be a knot with Alexander polynomial $A(z)=a+(b-a)z+(1-2b)z^2+(b-a)z^3+a z^4.$ Let $s(n)$ and $t(n)$ be the associated with $A(z)$ recurrent sequences. Suppose that $n$ is even. Consider the two following sequences $\hat\alpha(n)=\gcd(t(n),\frac{s(n)+(5a - b)t(n)}{2})$ and $\hat\beta(n)=\frac{s(n)^2+(4a-(3a+b)^2)t(n)^2}{4\alpha(n)}$ and set $k=1+4a-4b.$ Then $H_{1}(M_{2},\mathbb{Z})=\mathbb{Z}_{k}$ and we have the following exact sequence of Abelian groups
$$0 \longrightarrow \mathbb{A}\oplus\mathbb{A} \longrightarrow H_{1}(M_{n},\mathbb{Z}) \longrightarrow \mathbb{Z}_{k} \longrightarrow 0,$$ where $\mathbb{A}=\mathbb{Z}_{\hat\alpha(n)}\oplus\mathbb{Z}_{\hat\beta(n)}.$
\end{theorem}

\begin{proof}

Let $n$ be even. Direct calculation from formula (\ref{defineB}) leads to $B(2)=\Gamma^2-(\Gamma-I)^2=2\Gamma-I.$ Hence 
$$B(2)=\left(\begin{array}{cccc}
-1 + 2 a & -2 + 2 b& -2 b& -2 a\\ 2 a & -1 + 2 b & -2 b & -2 a\\ 
2 a & 2 b & 1 - 2 b & -2 a\\ 2 a & 2 b & 2 - 2 b & 1 - 2a
\end{array}\right).$$  Its determinant  
is $1+4a-4b.$ Since $a$ and $b$ are whole numbers $B(2)$ is invertible. 

It easy to see that the greatest common divisor of all entries of $B(2)$ is $1,$ The same is true for minors of orders $2$ and $3.$ So, the Smith Normal Form of $B(2)$ is $\textrm{diag}\{1,1,1,1+4a-4b\}$ and $H_{1}(M_{2},\mathbb{Z})=\textrm{coker}(B(2))={Z}_{1+4a-4b}.$ 

By making use of Lemma \ref{evenformB} we have 
$$B(2)^{-1}B(n)=\left(\begin{array}{cccc}
\alpha& \beta & \gamma & \delta\\ 
-\delta& \alpha-\gamma& -\delta& 0\\ 
0 & -\delta & \alpha-\gamma &-\delta\\ 
\delta & \gamma & \beta &\alpha
\end{array}\right),$$ where $\alpha=\frac{s(n)+(a-b)t(n)}{2}, \beta=(-1+a+2b)t(n), \gamma=(a-b)t(n), \delta=-a\,t(n).$ We note for $n$ odd, the matrix of this form appeared early in Theorem \ref{HomoddB}. 

From here we consider two matrices $W_1=B(2)$ and $W_2=B(2)^{-1}B(n)$ as $\mathbb{Z}$-linear operators from $\mathbb{Z}^4$ into itself. Dealing with $W_2$ in the same way as in the proof of Theorem \ref{HomoddB} we deduce that the Abelian group represented by $W_2$ has the form $\mathbb{A}\oplus\mathbb{A},$ where $\mathbb{A}=\mathbb{Z}_{\hat\alpha(n)}\oplus\mathbb{Z}_{\hat\beta(n)}.$ Here $\hat\alpha(n)$ and $\hat\beta(n)$ are the same as in the statement under proof.

Moreover, there exist an exact sequence of the type 
$$ker\, W_1 \longrightarrow coker\,W_2 \longrightarrow coker\, W_1 W_2 \longrightarrow coker\,W_1 \longrightarrow 0.$$ 
For the proof see, for example (\cite{SSK}, Snowflake Lemma 2.3.16)  or (\cite{Huang}, Theorem 2.1). 

From the above we have $\textrm{ker}\,W_1=0$ and $\textrm{coker}\,W_1= \mathbb{Z}_{1+4a-4b}.$ Also, $\textrm{coker}\,W_2=\mathbb{A}\oplus\mathbb{A}$ and $\textrm{coker}\,W_1 W_2=H_{1}(M_{n},\mathbb{Z}).$ Hence, the previous exact sequence morphs into 

$$0 \longrightarrow \mathbb{A}\oplus\mathbb{A} \longrightarrow H_{1}(M_{n},\mathbb{Z}) \longrightarrow \mathbb{Z}_{k} \longrightarrow 0.$$
\end{proof}

\section{Examples}

\subsection{Knot $6_2$}

Alexander polynomial: $A(z)=-1+3z-3z^2+3z^3-z^4.$ Here $a=-1,b=2.$

We set $\alpha(n)=\gcd(t(n),\frac{-7t(n)+s(n)}{2})$ and $\beta(n)=\frac{s(n)^2-5t(n)^2}{4\alpha(n)}.$

If $n$ is odd then $H_1(M_n(6_3),\mathbb{Z})= \mathbb{Z}_{\alpha(n)}^2\oplus\mathbb{Z}_{\beta(n)}^2.$

If $n$ is even and is not a multiple of $12$ then 
$$H_1(M_n(6_3),\mathbb{Z})=\mathbb{Z}_{\alpha(n)}\oplus\mathbb{Z}_{\alpha(n)}\oplus\mathbb{Z}_{\beta(n)}\oplus 
\mathbb{Z}_{11\beta(n)}$$ otherwise $$H_1(M_n(6_3),\mathbb{Z})=\mathbb{Z}_{\alpha(n)}\oplus\mathbb{Z}_{11\alpha(n)}\oplus\mathbb{Z}_{\beta(n)}\oplus\mathbb{Z}_{\beta(n)}.$$

Here is a table for initial values of $\alpha(n)$ and $\beta(n).$

\begin{table}[htp]
\caption{Data for $H_1(M_n(6_2),\mathbb{Z})$}
\begin{center}
\begin{tabular}{|c|c|c|c|c|c|c|c|c|c|c|c|c|}
\hline
$n$&1&2&3&4&5&6&7&8&9&10&11&12 \\
\hline
$\alpha(n)$&1&1&1&1&2&1&1&3&1&4&1&1 \\
\hline
$\beta(n)$&1&1&5&1&2&5&29&3&5&4&131&55 \\
\hline
\end{tabular}
\end{center}
\label{default}
\end{table}%

\subsection{Knot $6_3$}

Alexander polynomial: $A(z)=1-3z+5z^2-3z^3+z^4.$ Here $a=1,b=-2.$

We set $\alpha(n)=\gcd(t(n),\frac{3t(n)+s(n)}{2})$ and $\beta(n)=\frac{s(n)^2-21t(n)^2}{4\alpha(n)}.$

If $n$ is odd then $H_1(M_n(6_3),\mathbb{Z})= \mathbb{Z}_{\alpha(n)}^2\oplus\mathbb{Z}_{\beta(n)}^2.$

If $n$ is even and is not a multiple of $14$ then 
$$H_1(M_n(6_3),\mathbb{Z})=\mathbb{Z}_{\alpha(n)}\oplus\mathbb{Z}_{\alpha(n)}\oplus\mathbb{Z}_{\beta(n)}\oplus 
\mathbb{Z}_{13\beta(n)}$$ otherwise $$H_1(M_n(6_3),\mathbb{Z})=\mathbb{Z}_{\alpha(n)}\oplus\mathbb{Z}_{13\alpha(n)}\oplus\mathbb{Z}_{\beta(n)}\oplus\mathbb{Z}_{\beta(n)}.$$

Here is a table for initial values of $\alpha(n)$ and $\beta(n).$

\begin{table}[htp]
\caption{Data for $H_1(M_n(6_3),\mathbb{Z})$}
\begin{center}
\begin{tabular}{|c|c|c|c|c|c|c|c|c|c|c|c|c|c|c|}
\hline
$n$&1&2&3&4&5&6&7&8&9&10&11&12&13&14 \\
\hline
$\alpha(n)$&1&1&1&1&4&1&1&1&1&8&1&3&1&1 \\
\hline
$\beta(n)$&1&1&7&3&4&7&43&21&133&8&397&63&1171&559 \\
\hline
\end{tabular}
\end{center}
\label{default}
\end{table}%

\subsection{Knot $7_7$}

Alexander polynomial: $A(z)=1-5z+9z^2-5z^3+z^4.$ Here $a=1,b=-4.$

We set $\alpha(n)=\gcd(t(n),\frac{9t(n)+s(n)}{2})$ and $\beta(n)=\frac{s(n)^2+3t(n)^2}{4\alpha(n)}.$

If $n$ is odd then $H_1(M_n(7_7),\mathbb{Z})= \mathbb{Z}_{\alpha(n)}^2\oplus\mathbb{Z}_{\beta(n)}^2.$

If $n$ is even we have the following cases  
\begin{enumerate}
\item $n\equiv0\mod12$
$$H_1(M_n(7_7),\mathbb{Z})=\mathbb{Z}_{\alpha(n)}\oplus\mathbb{Z}_{21\alpha(n)}\oplus\mathbb{Z}_{\beta(n)}\oplus 
\mathbb{Z}_{\beta(n)};$$ 
\item $n\equiv2,10\mod12$
$$H_1(M_n(7_7),\mathbb{Z})=\mathbb{Z}_{\alpha(n)}\oplus\mathbb{Z}_{\alpha(n)}\oplus\mathbb{Z}_{\beta(n)}\oplus 
\mathbb{Z}_{21\beta(n)};$$ 
\item $n\equiv4,8\mod12$
$$H_1(M_n(7_7),\mathbb{Z})=\mathbb{Z}_{\alpha(n)}\oplus\mathbb{Z}_{7\alpha(n)}\oplus\mathbb{Z}_{\beta(n)}\oplus\mathbb{Z}_{3\beta(n)};$$
\item $n\equiv6\mod12$
$$H_1(M_n(7_7),\mathbb{Z})=\mathbb{Z}_{\alpha(n)}\oplus\mathbb{Z}_{3\alpha(n)}\oplus\mathbb{Z}_{\beta(n)}\oplus\mathbb{Z}_{7\beta(n)}.$$
\end{enumerate}

Here is a table for initial values of $\alpha(n)$ and $\beta(n).$

\begin{table}[htp]
\caption{Data for $H_1(M_n(7_7),\mathbb{Z})$}
\begin{center}
\begin{tabular}{|c|c|c|c|c|c|c|c|c|c|c|c|c|}
\hline
       $n$ & 1 & 2 & 3 & 4 & 5 & 6 & 7  & 8  & 9  & 10 & 11  & 12\\
\hline
$\alpha(n)$& 1 & 1 & 1 & 1 & 2 & 1 & 1  & 1  & 1  & 8  & 1   & 5\\
\hline
$\beta(n)$ & 1 & 1 & 13 & 7 & 38 & 39 & 421 & 217 & 2353 & 152 & 13201 & 1365\\
\hline
\end{tabular}
\end{center}
\label{default}
\end{table}%

\section*{Acknowledgment} 
The  author  was supported by Mathematical Center in Akademgorodok under agreement No. 075-15-2019-1613 with the Ministry of Science and Higher Education of the Russian Federation.

\end{document}